\documentclass[reqno]{amsart}
\usepackage[utf8]{inputenc}
\usepackage{amssymb, amsmath, amsthm, color, enumerate, mathabx, mathrsfs}
\usepackage{bbm}
\usepackage{pdfpages}
\usepackage{tcolorbox}
\usepackage{esint} % for \fint

\numberwithin{equation}{section}
\theoremstyle{plain}
\newtheorem{theorem}{Theorem}[section]
\newtheorem{lemma}[theorem]{Lemma}
\newtheorem{proposition}[theorem]{Proposition}
\newtheorem{corollary}[theorem]{Corollary}

\theoremstyle{definition}

\newtheorem{remark}[theorem]{Remark}

\newcommand{\R}{\mathbb{R}}

\newcommand{\cC}{\mathcal{C}}

\newcommand{\cH}{\mathcal{H}}

\newcommand{\bdT}{\mathbf{T}}

\newcommand{\bde}{\mathbf{e}}

\newcommand{\bdt}{\mathbf{t}}

\newcommand{\bdx}{\mathbf{x}}

\newcommand{\bdtheta}{\boldsymbol{\theta}}

\newcommand{\bdTheta}{\boldsymbol{\Theta}}

\newcommand{\ud}{\mathrm{d}}

\newcommand{\dist}{\mathrm{dist}}

\begin{document}

\title[Spherical maximal estimates via geometry]{Spherical maximal estimates via geometry}

\date{}

\begin{abstract}
We present a simple geometric approach to studying the $L^p$ boundedness properties of Stein's spherical maximal operator, which does not rely on the Fourier transform. Using this, we recover a weak form of Stein's spherical maximal theorem. 
\end{abstract}

\author[J. Hickman]{ Jonathan Hickman }
\address{Jonathan Hickman: School of Mathematics and Maxwell Institute for Mathematical Sciences, James Clerk Maxwell Building, The King's Buildings, Peter Guthrie Tait Road, Edinburgh, EH9 3FD, UK.}
\email{jonathan.hickman@ed.ac.uk}

\author[A. Jan\v car]{ Aj\v sa Jan\v car }
\address{Aj\v sa Jan\v car: School of Mathematics, James Clerk Maxwell Building, The King's Buildings, Peter Guthrie Tait Road, Edinburgh, EH9 3FD, UK.}
\email{a.jancar@sms.ed.ac.uk}

\maketitle

%%%%%%%%%%%%%%%%%%%%%%%%%%%%%%%%%%%%%%%%%%%%%%%%%%%%%%%%%%%%%%%%%%%%%%%%%%%%%%%%%%%%%%%%%%%%%%%%

%    Introduction

%%%%%%%%%%%%%%%%%%%%%%%%%%%%%%%%%%%%%%%%%%%%%%%%%%%%%%%%%%%%%%%%%%%%%%%%%%%%%%%%%%%%%%%%%%%%%%%%

\section{Introduction}

%%%%%%%%%%%%%%%%%%%%%%%%%%%%%%%%%%%%%%%%%%%%%%%%%%%%%%%%%%%%%%%%%%%%%%%%%%%%%%%%%%%%%%%%%%%%%%%%

%    The spherical maximal theorem

%%%%%%%%%%%%%%%%%%%%%%%%%%%%%%%%%%%%%%%%%%%%%%%%%%%%%%%%%%%%%%%%%%%%%%%%%%%%%%%%%%%%%%%%%%%%%%%%

\subsection{The spherical maximal theorem} Let $n \geq 2$ and, for $x \in \R^n$ and $r > 0$, let $C(x,r)$ denote the (Euclidean) sphere in $\R^n$ centred at $x$ of radius $r$. Furthermore, let $\sigma_{x,r}$ denote the spherical surface measure on $C(x,r)$, normalised to have mass one. Define the \textit{spherical maximal function}
\begin{equation}\label{eq: max op}
  Mf(x) := \sup_{r > 0}  \int_{C(x,r)} |f| \,\ud \sigma_{x,r}, \qquad x \in \R^n,
\end{equation}
acting (at least initially) on continuous, compactly supported functions $f \in C_c(\R^n)$. Thus, $M$ takes maximal averages of $f$ over the concentric spheres $C(x,r)$ centred at some fixed point $x \in \R^n$. 

An important result of Stein \cite{Stein1976} determines the $L^p$-mapping properties of $M$.

\begin{theorem}[Spherical maximal theorem \cite{Stein1976}]\label{thm: Stein} For all $n \geq 3$ and $p > p_n := \tfrac{n}{n-1}$, there exists a constant $C_{n,p} > 0$ such that the inequality
\begin{equation*}
    \|Mf\|_{L^p(\R^n)} \leq C_{n,p} \|f\|_{L^p(\R^n)}
\end{equation*}
holds for all $f \in C_c(\R^n)$.
\end{theorem}

Note the restriction $n \geq 3$. The result is still valid in the case $n=2$, but is significantly more difficult to prove and was resolved much later in a celebrated work of Bourgain~\cite{Bourgain1986}. In all dimensions, simple examples show that the range of estimates is sharp, in the sense that the $L^p$ bounds fail whenever $p \leq p_n$.

%%%%%%%%%%%%%%%%%%%%%%%%%%%%%%%%%%%%%%%%%%%%%%%%%%%%%%%%%%%%%%%%%%%%%%%%%%%%%%%%%%%%%%%%%%%%%%%%

%    A geometric approach to the spherical maximal theorem

%%%%%%%%%%%%%%%%%%%%%%%%%%%%%%%%%%%%%%%%%%%%%%%%%%%%%%%%%%%%%%%%%%%%%%%%%%%%%%%%%%%%%%%%%%%%%%%%

\subsection{A geometric approach to the spherical maximal theorem} Stein's  proof of Theorem~\ref{thm: Stein} \cite{Stein1976} (and, as far as we are aware, all subsequent proofs) is heavily Fourier analytical, exploiting decay properties of the Fourier transform $\widehat{\sigma}_{x,r}$ and Plancherel's theorem to prove favourable $L^2$ estimates.  In this note, we investigate a purely geometric approach to studying the spherical maximal operator.\medskip

Here we work with $\delta$-discretised maximal operators. Given $0 < \delta < 1/2$, $x \in \R^n$ and $1 \leq r \leq 2$, let $C^{\delta}(x,r)$ denote the annulus with outer radius $r$ and thickness $\delta$, formed by the $\delta$-neighbourhood of $C(x,r)$. We then define the family of auxiliary maximal operators
\begin{equation*}
  M^{\delta}f(x) := \sup_{1 \leq r \leq 2}  \fint_{C^{\delta}(x,r)} |f|, \qquad x \in \R^n,
\end{equation*}
acting on $f \in L^1_{\mathrm{loc}}(\R^n)$. Here, given a set $E \subseteq \R^n$ of finite, positive $n$-dimensional Lebesgue measure and a non-negative measurable $f$, we let $\fint_E f := \tfrac{1}{|E|} \int_E f$ denote the average of $f$ over $E$. 

Our goal is to prove the following weak variant of Theorem~\ref{thm: Stein}, which roughly tells us that $M$ is `almost' bounded on $L^p$ for $p \geq \frac{n}{n-1}$.

\begin{theorem}\label{thm: delta disc} For all $n \geq 2$, $p \geq p_n$, there exists a constant $C_{n, p} \geq 1$ such that inequality
\begin{equation*}
    \|M^{\delta}f\|_{L^p(\R^n)} \leq C_{n, p}(\log\delta^{-1})^{p_n/p}\|f\|_{L^p(\R^n)}
\end{equation*}
holds for all $f \in L^1_{\mathrm{loc}}(\R^n)$ and all $0 < \delta < 1/2$. 
\end{theorem}

This result is sharp in the sense that for $p < p_n$ the $L^p$ operator norm of $M^{\delta}$ cannot be bounded by some fixed power of $\log \delta^{-1}$ but must grow polynomially in $\delta^{-1}$. However, Theorem~\ref{thm: delta disc} is weaker than Theorem~\ref{thm: Stein} in two respects. Because of the $(\log\delta^{-1})$-losses, we do not recover the boundedness of the spherical maximal operator on any $L^p$ space. Furthermore, here we only consider the local operator, where the radii are constrained to satisfy $1 \leq r \leq 2$. Thus, here the interest is not in the result itself, but in the method of proof. As mentioned above, our argument is purely geometric, involving a careful analysis of the intersection properties of spherical annuli. Whilst it may be possible to push the methods further to recover the full strength of Theorem~\ref{thm: Stein}, arguably there is value in having a short, simple argument at the cost of a slightly weaker estimate. 

%%%%%%%%%%%%%%%%%%%%%%%%%%%%%%%%%%%%%%%%%%%%%%%%%%%%%%%%%%%%%%%%%%%%%%%%%%%%%%%%%%%%%%%%%%%%%%%%

%    Fourier analytic vs geometric arguments

%%%%%%%%%%%%%%%%%%%%%%%%%%%%%%%%%%%%%%%%%%%%%%%%%%%%%%%%%%%%%%%%%%%%%%%%%%%%%%%%%%%%%%%%%%%%%%%%

\subsection{Fourier analytic vs geometric arguments} Fourier analytic methods are very powerful in the theory of geometric maximal functions. Rubio de Francia~\cite{R1986} established a far-reaching extension of Theorem~\ref{thm: Stein} for a class of maximal Fourier multiplier operators, which allows one to study maximal averages over a large variety of surfaces and also fractal sets. Bourgain's celebrated proof of the circular maximal operator relies heavily on Fourier analysis~\cite{Bourgain1986}, as does an alternative proof due to Mockenhaupt--Seeger--Sogge~\cite{MSS1992}, which made an important connection between the circular maximal operator and the \textit{local smoothing} phenomenon for the wave equation. Fourier analysis has also played a substantial role in the modern theory; for instance, square functions and decoupling inequalities from Fourier restriction theory are central to the recent proofs of the $L^p$ boundedness of the helical maximal operator \cite{BGHS, KLO2022}.

Despite the power of Fourier analytical tools, it is also of interest to understand the maximal operator such as \eqref{eq: max op} from a purely geometric and combinatorial standpoint, without making use of the Fourier transform. An early instance of this is Marstrand's circle packing theorem~\cite{Marstrand1987}, which is a precursor to the circular maximal theorem. Marstrand's argument combines geometric information about circle tangencies with a variant of the K\"ovari--S\'os--Tur\'an argument \cite{KST1954} from combinatorics. These ideas were further developed in the 1990s in the context of maximal functions by Kolasa--Wolff~\cite{KW1999}, Wolff~\cite{Wolff1999, Wolff2000} and Schlag~\cite{Schlag1997, Schlag1998}. In particular, in a remarkable and intricate paper, Schlag~\cite{Schlag1998} gave an entirely geometric proof of Bourgain's circular maximal function theorem. 

Geometric arguments offer a different perspective on maximal operators and in some cases can be more effective than Fourier analytic methods. For instance, the aforementioned works \cite{KW1999, Wolff1999} treat multiparameter circular maximal operators (where one takes the supremum over the possible centres of the circles, rather than the radii) where Fourier analytic techniques appear to be less effective than in the one parameter case. Furthermore, many recent advances in the theory of geometric maximal operators rely on a synthesis of both geometric and Fourier analytic tools: see, for instance, the relationship between the works \cite{Zahl, CGY, LLO2025}.\medskip

%%%%%%%%%%%%%%%%%%%%%%%%%%%%%%%%%%%%%%%%%%%%%%%%%%%%%%%%%%%%%%%%%%%%%%%%%%%%%%%%%%%%%%%%%%%%%%%%

%    Notational conventions

%%%%%%%%%%%%%%%%%%%%%%%%%%%%%%%%%%%%%%%%%%%%%%%%%%%%%%%%%%%%%%%%%%%%%%%%%%%%%%%%%%%%%%%%%%%%%%%%

\subsection{Notational conventions}\label{subsec: notation}  Given a list of objects $L$ and non-negative real numbers $A$, $B$, we write $A \lesssim_L B$ or $B \gtrsim_L A$ if $A \leq C_L B$ where $C_L > 0$ is a constant depending only on the objects listed in $L$ and a choice of dimension $n$. Given a vector space $V$, we let $\langle x_1, \dots, x_m \rangle$ denote the linear span of $x_1, \dots, x_m \in V$. If $V$ is an inner product space and $E \subseteq V$ is a vector subspace, then we let $E^{\perp}$ denote the orthogonal complement and $\mathrm{proj}_E \colon V \to E$ the orthogonal projection map. \medskip

%%%%%%%%%%%%%%%%%%%%%%%%%%%%%%%%%%%%%%%%%%%%%%%%%%%%%%%%%%%%%%%%%%%%%%%%%%%%%%%%%%%%%%%%%%%%%%%%

%    Acknowledgements

%%%%%%%%%%%%%%%%%%%%%%%%%%%%%%%%%%%%%%%%%%%%%%%%%%%%%%%%%%%%%%%%%%%%%%%%%%%%%%%%%%%%%%%%%%%%%%%%

\noindent \textbf{Acknowledgements.} The first author is supported by New Investigator Award UKRI097. He wishes to thank David Beltran and Joshua Zahl for helpful comments and encouragement regarding an early draft of this work. This research was initiated in summer 2023 when the second author was supported by a University of Edinburgh School of Mathematics Vacation Scholarship. It also forms part of her undergraduate dissertation at the University of Edinburgh.

%%%%%%%%%%%%%%%%%%%%%%%%%%%%%%%%%%%%%%%%%%%%%%%%%%%%%%%%%%%%%%%%%%%%%%%%%%%%%%%%%%%%%%%%%%%%%%%%

%    Proof strategy

%%%%%%%%%%%%%%%%%%%%%%%%%%%%%%%%%%%%%%%%%%%%%%%%%%%%%%%%%%%%%%%%%%%%%%%%%%%%%%%%%%%%%%%%%%%%%%%%

\section{Summary of the proof}

%%%%%%%%%%%%%%%%%%%%%%%%%%%%%%%%%%%%%%%%%%%%%%%%%%%%%%%%%%%%%%%%%%%%%%%%%%%%%%%%%%%%%%%%%%%%%%%%

%    An enemy scenario

%%%%%%%%%%%%%%%%%%%%%%%%%%%%%%%%%%%%%%%%%%%%%%%%%%%%%%%%%%%%%%%%%%%%%%%%%%%%%%%%%%%%%%%%%%%%%%%%

\subsection{An enemy scenario}\label{subsec: enemy} Here we give an overview of the key ideas in the proof of Theorem~\ref{thm: delta disc}. We focus on $n=3$ case, which is already representative of the general situation. Here the critical exponent is $p_n = 3/2$. By a standard duality argument \`a la C\'ordoba~\cite{Cordoba1977}, matters reduce to bounding expressions of the form
\begin{equation}\label{eq: multiplicity}
    \big\|\sum_{C \in \cC} \chi_{C^{\delta}}\|_{L^3(\R^3)}^3 = \sum_{C_1, C_2, C_3 \in \cC} | C_1^{\delta} \cap C_2^{\delta} \cap C_3^{\delta}|,
\end{equation}
where $\cC$ is a family of spheres in $\R^3$ with $\delta$-separated centres and $C^{\delta}$ denotes the $\delta$-neighbourhood of $C \in \cC$.\medskip

Generically, we should expect three spheres $C_1$, $C_2$, $C_3 \in \cC$ to intersect transversally, leading to a volume bound
\begin{equation}\label{eq: good intersection}
   | C_1^{\delta} \cap C_2^{\delta} \cap C_3^{\delta}| \lesssim \delta^3.
\end{equation}
If the generic case held for every single triple, then we would immediately have a favourable estimate for \eqref{eq: multiplicity}. However, there are many non-generic situations which lead to larger sphere intersections.\medskip

\noindent 1) If the three spheres $C_1$, $C_2$, $C_3$ have collinear centres, then they can intersect along a common circle. This leads to the weaker volume bound 
\begin{equation}\label{eq: v bad intersection}
    | C_1^{\delta} \cap C_2^{\delta} \cap C_3^{\delta}| \lesssim \delta^2,
\end{equation}
assuming that the centres of the spheres are well-separated. To isolate and deal with this situation, we simply divide the summation over $C_1, C_2, C_3 \in \cC$ according to the angle formed between $x(C_2) - x(C_1)$ and $x(C_3) - x(C_1)$, where $x(C_j)$ denotes the centre of $C_j$. For small angles, we may have a bad intersection bound as in \eqref{eq: v bad intersection}, but this is compensated by the fact there are relatively few triples of spheres which form a small angle. This part of the proof is essentially a variant C\'ordoba's argument for bounding the Kakeya maximal operator~\cite{Cordoba1977}. \medskip

\noindent 2) Collinearity of centres is not the only enemy. Suppose we have a triple of spheres $C_1$, $C_2$, $C_3$ whose centres are well-separated and very far from colinear in the sense that the angle formed between $x(C_2) - x(C_1)$ and $x(C_3) - x(C_1)$ is large. In this situation, we would ideally like to have a good intersection bound as in \eqref{eq: good intersection}. However, there is a further enemy scenario, which is illustrated in Figure~\ref{fig: enemy}. The sets $C_1 \cap C_2$ and $C_1 \cap C_3$ are circles on the sphere $C_1$. If the radii of $C_2$ and $C_3$ are suitably chosen, then $C_1 \cap C_2$ and $C_1 \cap C_3$ can be tangent. This leads to an unfavourable volume bound of 
\begin{equation}\label{eq: bad intersection}
    |C_1^{\delta} \cap C_2^{\delta} \cap C_3^{\delta}| \sim \delta^{5/2}.
\end{equation}
It is not possible to compensate for the poor bound in \eqref{eq: bad intersection} by a simple counting argument as in 1). Indeed, once we fix $C_1$, it is possible to set things up so the enemy scenario happens for essentially all choices of $C_2$ and $C_3$.\medskip

We remark that the enemy scenario described in 2) also plays a role in the analysis of certain Nikodym-type maximal operators associated to spheres: see~\cite{CDK}. However, for the operators studied in \cite{CDK}, one can work directly with the $\delta^{5/2}$ numerology and still obtain sharp bounds, which is not the case for the classical spherical maximal function. 

\begin{figure}
    \centering
    \includegraphics[width=0.7\linewidth]{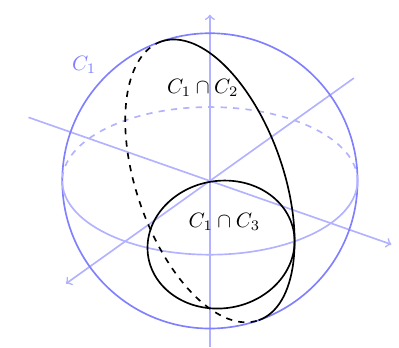}
    \caption{The enemy scenario for $n = 3$. Here we have a fixed sphere $C_1$ which is intersected by two further spheres $C_2$ and $C_3$ (not shown in the figure). The intersections $C_1 \cap C_2$ and $C_1 \cap C_3$ form a pair of \textit{tangent} circles on $C_1$, leading to an unfavourable volume bound of $|C_1^{\delta} \cap C_2^{\delta} \cap C_3^{\delta}| \sim \delta^{5/2}$. Note if the centres of $C_1$, $C_2$, $C_3$ all lie on the horizontal coordinate plane, then the point of tangency necessarily lies on the equator of $C_1$.}
    \label{fig: enemy}
\end{figure}

%%%%%%%%%%%%%%%%%%%%%%%%%%%%%%%%%%%%%%%%%%%%%%%%%%%%%%%%%%%%%%%%%%%%%%%%%%%%%%%%%%%%%%%%%%%%%%%%

%    Circumventing the enemy scenario

%%%%%%%%%%%%%%%%%%%%%%%%%%%%%%%%%%%%%%%%%%%%%%%%%%%%%%%%%%%%%%%%%%%%%%%%%%%%%%%%%%%%%%%%%%%%%%%%

\subsection{Circumventing the enemy scenario}\label{subsec: avoid enemy}

The situation discussed in 2) shows that circle tangencies are an important consideration in the analysis of the spherical maximal operator in $\R^3$ around the critical $3/2$ exponent. This should be compared to the role of circle tangencies in the study of Bourgain's circular maximal function around the critical $2$ exponent. \medskip

Rather than treating circle tangencies directly, our strategy is to use a \textit{variable slicing} argument to circumvent the enemy scenario. Variable slicing appears to be a powerful tool in the study of geometric maximal functions and has previously been applied in related contexts (see, for instance, \cite[\S6]{Zahl}).\medskip

Suppose we have a triple of spheres $C_1$, $C_2$, $C_3$ such that $C_1 \cap C_2$ and $C_1 \cap C_3$ form a pair of tangent circles on $C_1$, as in the enemy scenario described in 2). Further suppose that the centres of $C_1$, $C_2$ and $C_3$ lie on the horizontal plane $\R^2 \times \{0\}$. The key observation is that the point of tangency between $C_1 \cap C_2$ and $C_1 \cap C_3$ must also lie on the plane $\R^2 \times \{0\}$: see Figure~\ref{fig: enemy}.\medskip

To exploit the above observation, we first use the rotation invariance of the problem to replace the entire sphere $C(x,r)$ in the definition of the maximal operator with a small polar cap centred around the vertical axis. We then turn to the ostensibly harder problem of proving uniform bounds for the sliced norms
\begin{equation*}
   \sup_{a \in \R} \|M^{\delta}f\|_{L^p(\R^2 \times \{a\})}
\end{equation*}
of the maximal function (now defined with respect to polar caps). Since our operator is local, such a sliced bound easily leads to a true $L^p(\R^3)$ estimate. By translation invariance, we may further assume $a = 0$.\medskip

When we apply the duality argument to this sliced formulation, we are led to consider a variant of \eqref{eq: multiplicity} where now all the spheres in $\cC$ are centred on $\R^2 \times \{0\}$ and each $C^{\delta}$ is replaced with $C^{\delta, \star}$, which represents the $\delta$-neighbourhood of a polar cap on $C$ (rather than a whole sphere). This means that, given a triple $C_1$, $C_2$, $C_3$ of spheres, a point of tangency between $C_1 \cap C_2$ and $C_1 \cap C_3$ (if it exists) must also lie on the plane $\R^2 \times \{0\}$. However, we are only considering the polar regions $C^{\delta, \star}$, which lie very far from $\R^2 \times \{0\}$, so any such tangencies fail to contribute. 

%%%%%%%%%%%%%%%%%%%%%%%%%%%%%%%%%%%%%%%%%%%%%%%%%%%%%%%%%%%%%%%%%%%%%%%%%%%%%%%%%%%%%%%%%%%%%%%%

%    Proof of the Theorem~\ref{thm: delta disc}

%%%%%%%%%%%%%%%%%%%%%%%%%%%%%%%%%%%%%%%%%%%%%%%%%%%%%%%%%%%%%%%%%%%%%%%%%%%%%%%%%%%%%%%%%%%%%%%%

\section{Proof of Theorem~\ref{thm: delta disc}}

%%%%%%%%%%%%%%%%%%%%%%%%%%%%%%%%%%%%%%%%%%%%%%%%%%%%%%%%%%%%%%%%%%%%%%%%%%%%%%%%%%%%%%%%%%%%%%%%

%    Initial reductions

%%%%%%%%%%%%%%%%%%%%%%%%%%%%%%%%%%%%%%%%%%%%%%%%%%%%%%%%%%%%%%%%%%%%%%%%%%%%%%%%%%%%%%%%%%%%%%%%

\subsection{Initial reductions} We begin by implementing the reduction to polar caps and the slicing argument discussed in \S\ref{subsec: avoid enemy}. Given $x \in \R^n$ and $r > 0$, recall that
\begin{equation*}
    C^{\delta}(x,r)  = \big\{ y \in \R^n : ||y - x| - r|<\delta \big\}.
\end{equation*}
We let $C^{\delta, \star}(x,r)$ denote the polar region
\begin{equation*}
 C^{\delta, \star}(x,r)  = \big\{ y \in \R^n : ||y - x| - r|<\delta, \, |y_n - x_n| < r/100n\big\}.   
\end{equation*}
By pigeonholing and the rotational symmetry of the problem, it suffices to show Theorem~\ref{thm: delta disc} holds with $M^{\delta}$ replaced with
\begin{equation*}
  M^{\delta, \star}f(x) := \sup_{1 \leq r \leq 2}  \fint_{C^{\delta, \star}(x,r)} |f|, \qquad x \in \R^n.
\end{equation*}

By interpolation with the trivial $L^{\infty}$ estimate, the full range of estimates follows once we have established the endpoint case $p = p_n$. Furthermore, since the operator is local, it suffices to show 
\begin{equation*}
    \|M^{\delta, \star}f\|_{L^{p_n}(Q^n)} \lesssim (\log \delta^{-1})\|f\|_{L^{p_n}(\R^n)}
\end{equation*}
where $Q := [-2^{-1}n^{-1/2}, 2^{-1} n^{-1/2}]$ so that the domain $Q^n$ has diameter $1$. Finally, by Fubini's theorem and translation invariance, the problem is further reduced to proving the bound 
\begin{equation}\label{eq: red max 1}
    \|M^{\delta, \star}f\|_{L^{p_n}(Q^{n-1} \times \{0\})} \lesssim  (\log \delta^{-1}) \|f\|_{L^{p_n}(\R^n)}.
\end{equation}

Although elementary, the combination of the reduction to polar caps $C^{\delta, \star}$ and slicing is essential for the following argument. 
%%%%%%%%%%%%%%%%%%%%%%%%%%%%%%%%%%%%%%%%%%%%%%%%%%%%%%%%%%%%%%%%%%%%%%%%%%%%%%%%%%%%%%%%%%%%%%%%

%    Multiplicity formulation

%%%%%%%%%%%%%%%%%%%%%%%%%%%%%%%%%%%%%%%%%%%%%%%%%%%%%%%%%%%%%%%%%%%%%%%%%%%%%%%%%%%%%%%%%%%%%%%%

\subsection{Multiplicity formulation}

We apply discretisation and duality arguments to convert the problem into a multiplicity bound. In particular, \eqref{eq: red max 1} is equivalent to the following. 

\begin{proposition}\label{prop: multiplicity} For all $0 < \delta < 1$, we have
    \begin{equation*}
    \big\|\sum_{C \in \cC}  \chi_{C^{\delta, \star}}\|_{L^n(\R^n)} \lesssim (\log\delta^{-1})^{1/n} \delta^{1-(n-1)^2/n} [\#\cC]^{1/n}
\end{equation*}
whenever $\cC$ is a family of spheres with $\delta$-separated centres lying in $Q^{n-1} \times \{0\}$ and with radii lying in $[1,2]$.
\end{proposition}

This reduction is standard within the geometric maximal function literature; see \cite{Carbery1988} or \cite[Proposition 22.4]{Mattila_book} and \cite[Proposition 22.6]{Mattila_book} for an exposition of the argument in the Kakeya case, which easily adapts to the present setting. We remark that here we use an effective variant of the pigeonholing argument in \cite[Proposition 22.6]{Mattila_book}, involving an additional H\"older step, which results in a $(\log \delta^{-1})^{(n-1)/n}$ loss.

\subsection{Spherical intersection bound} For the remainder of the paper we focus on the proof of Proposition~\ref{prop: multiplicity}. Henceforth, let $\cC$ denote a family of spheres with $\delta$-separated centres lying in $Q^{n-1} \times \{0\}$ and with radii lying in $[1,2]$.\medskip

Observe that 
\begin{equation*}
    \big\|\sum_{C \in \cC}  \chi_{C^{\delta, \star}}\|_{L^n(\R^n)}^n =  \sum_{C_1, \dots, C_n \in \cC} \Big| \bigcap_{j=1}^{n} C_j^{\delta, \star} \Big|.
\end{equation*}
The crux of our argument is to understand the possible volume of the intersection between $n$ polar regions $C_1^{\delta, \star}, \dots, C_n^{\delta, \star}$. This volume is determined by the distance and angular separation between the centres of the underlying spheres. To measure these quantities, given any sphere $C \subset \R^n$ we let $x(C) \in \R^n$ denote the centre of $C$. For a pair of spheres $C$, $\bar{C} \subset \R^n$, we define
\begin{equation}\label{eq: sphere data}
\dist(C, \bar{C}) := |x(C) - x(\bar{C})| \qquad \textrm{and} \qquad  \bde(C, \bar{C}) := \dfrac{x(\bar{C})-x(C)}{|x(\bar{C})-x(C)|}, 
\end{equation}
where we only consider $\bde(C, \bar{C})$ in the case where $x(C) \neq x(\bar{C})$. If the centres of $C$, $\bar{C}$ are distinct and lie on $\R^{n-1} \times \{0\}$ (as is the case for $C$, $\bar{C} \in \cC$), then we let $e(C, \bar{C}) \in \R^{n-1}$ be the unit vector satisfying $\bde(C, \bar{C}) = (e(C, \bar{C}), 0)$.

For $\delta \leq t_2 \leq 1$ and $C_1 \in \cC$, we define
\begin{equation}\label{eq: C2 set}
  \cC_{t_2}^{C_1} := \Big\{C_2 \in \cC : \frac{t_2}{2} \leq \dist(C_1, C_2) \leq t_2\Big\}.
\end{equation}
Furthermore, for $3 \leq j \leq n$, $\delta \leq t_j \leq 1$ and $\frac{\delta}{t_j} \leq \theta_j \leq 1$ and a family of spheres $C_1, \hdots, C_{j-1} \in \cC$, we let $\cC_{t_j, \theta_j}^{C_1, \hdots, C_{j-1}}$ denote the set of all $C_j \in \cC$ satisfying\footnote{The (standard) notation used here is explained in \S\ref{subsec: notation}.}
\begin{equation}\label{eq: pgnhole dist angle}
\frac{t_j}{2} \leq \text{dist}(C_1, C_j) \leq t_j \quad \textrm{and} \quad \frac{\theta_j}{2} \leq |\operatorname{proj}_{\langle e(C_1, C_2), \hdots, e(C_1, C_{j-1}) \rangle ^{\perp}} e(C_1, C_j) | \leq \theta_j. 
\end{equation}
We also let $\cC_{t_j, \leq \delta/t_j}^{C_1, \hdots, C_{j-1}}$ denote the set of all $C_j \in \cC$ satisfying
\begin{equation}\label{eq: angular degenerate}
\frac{t_j}{2} \leq \text{dist}(C_1, C_j) \leq t_j \quad \textrm{and} \quad |\operatorname{proj}_{\langle e(C_1, C_2), \hdots, e(C_1, C_{j-1}) \rangle^{\perp}} e(C_1, C_j) | \leq \frac{2\delta}{t_j},
\end{equation}
although we shall only require these sets later in the argument. The key intersection bound is then provided by the following lemma.

\begin{lemma}\label{lem: generaltuplebound}
    Let $2 \leq m \leq n$ and $\delta \leq t_j \leq 1$ for $2 \leq j \leq m$ and $\frac{\delta}{t_j} \leq \theta_j \leq 1$ for $3 \leq j \leq m$. For $C_1 \in \cC$, $C_2 \in \cC_{t_2}^{C_1}$ and $C_j \in \cC_{t_j, \theta_j}^{C_1, \hdots, C_{j-1}}$ for $3 \leq j \leq m$, we have
    \begin{equation*}
        \Big| \bigcap_{j=1}^m C_j^{\delta, \star} \Big| \lesssim \dfrac{\delta^m}{\prod_{j=2}^m t_j \prod_{j=3}^m \theta_j}.
    \end{equation*}
\end{lemma}

\begin{remark} It is crucial to work with the polar regions $C_j^{\delta, \star}$, rather than the entire annuli $C_j^{\delta}$, centred on $\R^{n-1} \times \{0\}$. Indeed, otherwise the enemy scenario described in \S\ref{subsec: enemy} provides a counterexample to the bound in Lemma~\ref{lem: generaltuplebound}.    
\end{remark}

We begin by recalling the following elementary lemma governing intersections between pairs of spheres. 

\begin{lemma}\label{lem: pairs} Suppose $C$, $\bar{C} \subseteq \R^n$ are two spheres with distinct centres and radii lying in $[1,2]$. For $0 < \delta < 1/2$, we have $C^{\delta} \cap \bar{C}^{\delta} \subseteq \Pi^{\delta}(C, \bar{C})$ where $\Pi^{\delta}(C, \bar{C})$ is an $O(\frac{\delta}{\dist(C, \bar{C})})$-neighbourhood of an affine hyperplane $\Pi(C, \bar{C})$ which has normal direction $\bde(C, \bar{C})$.     
\end{lemma}

This result appears implicitly in the literature (for instance, in \cite[Lemma 2.1]{KW1999}). Nevertheless, for completeness, we present the simple proof.

\begin{proof}[Proof (of Lemma~\ref{lem: pairs})] We first consider the intersection between two spheres $C = C(x, r)$ and $\bar{C} = C(\bar{x}, \bar{r})$ where $r$, $\bar{r} \in [1, 2]$ and $x$, $\bar{x} \in \R^n$ are such that $x \neq \bar{x}$. Suppose there exists a point $y \in C \cap \bar{C}$. Then, by expanding the defining equations of the spheres, 
\begin{align*}
|y|^2- 2\langle y, \; x\rangle + |x|^2 = r^2\qquad \textrm{and} \qquad
|y|^2- 2\langle y, \; \bar{x}\rangle + |\bar{x}|^2 = \bar{r}^2.
\end{align*}
Subtracting the two equations, we get 
\begin{equation*}
    \langle \bde(C, \bar{C}), y \rangle = \frac{r^2 - \bar{r}^2 - (|x|^2-|\bar{x}|^2)}{2\dist(C, \bar{C})}
\end{equation*}
where $\dist(C, \bar{C})$ and $\bde(C, \bar{C})$ are as defined in \eqref{eq: sphere data}. Thus, we see that $C \cap \bar{C}$ lies in an affine hyperplane $\Pi(C, \bar{C})$ with normal equal to $\bde(C, \bar{C})$.\medskip

Now consider the continuum analogue of the above setup, replacing the spheres $C$, $\bar{C}$ with the annuli $C^{\delta}$, $\bar{C}^{\delta}$. If $y \in C^{\delta} \cap \bar{C}^{\delta}$, then there exist $r'$, $\bar{r}\,' \in [1/2, 4]$ satisfying $|r - r'|<\delta$ and $|\bar{r} - \bar{r}\,'|<\delta$ such that $y \in C(x, r') \cap C(\bar{x}, \bar{r}\,')$. We now apply the preceding observations to the pair of perturbed spheres $C(x, r')$, $C(\bar{x}, \bar{r}\,')$. From this, it follows that $C^{\delta} \cap \bar{C}^{\delta} \subseteq \Pi^{\delta}(C, \bar{C})$ where $\Pi^{\delta}(C, \bar{C})$ is an $O(\frac{\delta}{\dist(C, \bar{C})})$-neighbourhood of $\Pi(C, \bar{C})$.
\end{proof}

\begin{proof}[Proof (of Lemma~\ref{lem: generaltuplebound})]
    By Lemma~\ref{lem: pairs} and the constraints on the centres and radii of the spheres, we have 
    \begin{equation} \label{lem: tripleslemmagen1}
    C_1^{\delta, \star} \cap C_2^{\delta, \star} \cap \hdots \cap C_m^{\delta, \star} \subseteq \Pi_2^{\delta} \cap \Pi_3^{\delta} \cap \hdots \cap \Pi_m^{\delta} \cap [-10, 10]^n
    \end{equation}
    where each $\Pi^\delta_j$ is an $O(\frac{\delta}{t_j})$-neighbourhood of an affine hyperplane with normal $\bde(C_1, C_j)$. Since the spheres are all centred on the horizontal hyperplane $\R^{n-1} \times \{0\}$, each normal vector $\bde(C_1, C_j)$ lies on the horizontal hyperplane. We may therefore define $R$ to be the parallelepiped in $\R^{n-1}$ satisfying  
    $$
    \Pi^\delta_2 \cap \Pi^\delta_3 \cap \hdots \cap \Pi_m^{\delta} \cap \big([-10, 10]^{n-1} \times \{a\}\big) = R \times \{a\} \quad \textrm{for all $a \in \R$.} 
    $$
    In particular, the orthogonal projection of $\bigcap_{j=2}^m \Pi^\delta_j$ onto $\R^{n-1} \times \{0\}$ gives 
    \begin{equation} \label{lem: tripleslemmagen2}
        \operatorname{proj}_{e_n^\perp}\Big(\bigcap_{j=2}^m \Pi_j^{\delta} \cap [-10, 10]^n\Big) = R \times \{0\}.
    \end{equation}

   For each $2 \leq j \leq m$, we can write $\Pi^\delta_j = \widetilde{\Pi}^\delta_j \times \{a\}$ for all $a \in \R$ where $\widetilde{\Pi}^\delta_j$ is a neighbourhood of an affine hyperplane in $\R^{n-1}$ with thickness $O(\frac{\delta}{t_j})$ and normal $e(C_1, C_j)$ where $\bde(C_1, C_j) = (e(C_1, C_j), 0)$, as defined in the text following \eqref{eq: sphere data}. With this definition, 
   \begin{equation*}
        R = \bigcap_{j=2}^m \widetilde{\Pi}^\delta_j \cap [-10, 10]^{n-1},
   \end{equation*}
   and by applying a simple change of variables of integration
    \begin{equation}\label{eq: R vol 1}
      |R| \sim |e(C_1, C_2) \wedge \hdots \wedge e(C_1, C_m)|^{-1} \prod_{j=2}^m \frac{\delta}{t_j}.
    \end{equation}
    
    For $1 \leq \ell < d$ we now recall the determinant identity
    \begin{equation} \label{detproperty}
    |\bdx_1 \wedge  \hdots \wedge \bdx_{\ell}| = |\bdx_1|\prod_{j=2}^{\ell} \left| \operatorname{proj}_{\langle \bdx_1, \hdots, \bdx_{j-1} \rangle^{\perp}} \bdx_j \right|,
    \end{equation}
    for $\bdx_j \in \R^d$, $1 \leq j \leq \ell$, which generalises the familiar `base $\times$ height' formula for the area of a parallelogram to higher dimensions. Using \eqref{detproperty} and the hypothesis $C_j \in \cC_{t_j, \theta_j}^{C_1, \hdots, C_{j-1}}$ for $2 \leq j \leq m$, we may write
    \begin{equation}\label{eq: R vol 2}
        |e(C_1, C_2) \wedge \hdots \wedge e(C_1, C_m)|^{-1} \prod_{j=2}^m \frac{\delta}{t_j} \sim \Big(\prod_{j=3}^m \theta_j \Big)^{-1} \prod_{j=2}^m \frac{\delta}{t_j}.
    \end{equation}
    Therefore, combining \eqref{eq: R vol 1} and \eqref{eq: R vol 2}, we have
    \begin{equation} \label{lem: tripleslemmagen3}
        |R| \sim \dfrac{\delta^{m-1}}{\prod_{j=2}^m t_j \prod_{j=3}^m \theta_j}.
    \end{equation}

   Suppose $\ell \subset \R^n$ is a line parallel to the coordinate vector $e_n$ which intersects the polar region $C_1^{\delta, \star}$ at some point $y \in \R^n$. By the definition of the polar caps, the normal to $C_1$ at $y$ makes a small angle with the direction $e_n$ of the line $\ell$ and, crucially, $\ell$ is not tangent to $C_1$ at $y$. Therefore, $\ell$ is transversal to $C_1$ at $y$ and
    \begin{equation*}
        \cH^1\big(\ell \cap C_1^{\delta, \star}\big) \lesssim \delta.
    \end{equation*}
    Now letting $\ell(u)$ denote the line parallel to $e_n$ which passes through $\R^{n-1} \times \{0\}$ at $(u,0)$, using \eqref{lem: tripleslemmagen1} and \eqref{lem: tripleslemmagen2} with Fubini's theorem gives 
    \begin{align*}
        |C_1^{\delta, \star} \cap C_2^{\delta, \star} \cap \hdots \cap C_m^{\delta, \star}| &\leq |C_1^{\delta, \star} \cap\Pi^\delta_2 \cap \Pi^\delta_3 \cap \hdots \cap \Pi^\delta_m| \\
        &\leq \int\limits_{R \times \{0\}} \cH^1\big(\ell(u) \cap C_1^{\delta, \star}\big) \, \ud u \\
        &\lesssim \delta|R|.
    \end{align*}
     Finally, by \eqref{lem: tripleslemmagen3}, we have
    \begin{align*}
        |C^{\delta, \star}_1 \cap C^{\delta, \star}_2 \cap \hdots \cap C^{\delta, \star}_m| \lesssim \dfrac{\delta^m}{\prod_{j=2}^m t_j \prod_{j=3}^m \theta_j},
    \end{align*}
    which is our desired bound.
\end{proof}

%%%%%%%%%%%%%%%%%%%%%%%%%%%%%%%%%%%%%%%%%%%%%%%%%%%%%%%%%%%%%%%%%%%%%%%%%%%%%%%%%%%%%%%%%%%%%%%%

%        Bounding the cardinality of the sphere families

%%%%%%%%%%%%%%%%%%%%%%%%%%%%%%%%%%%%%%%%%%%%%%%%%%%%%%%%%%%%%%%%%%%%%%%%%%%%%%%%%%%%%%%%%%%%%%%%

\subsection{Bounding the cardinalities of the sphere families}%\label{subsec: cardinalitybound} 
Since the spheres in $\cC$ have $\delta$-separated centres lying in $\R^{n-1} \times \{0\}$, it immediately follows that the set $\cC_{t_2}^{C_1}$ introduced in \eqref{eq: C2 set} satisfies
\begin{equation*}
    \# \cC_{t_2}^{C_1} \lesssim (t_2 / \delta)^{n-1} \qquad \textrm{for all $C_1 \in \cC$ and $t_2 \geq \delta$.}
\end{equation*}
We next consider the cardinality of the sets $\cC^{C_1, \hdots, C_{j-1}}_{t_j, \theta_j}$ for $3 \leq j \leq n$. For these families, the desired bound is a direct consequence of the following volume estimate. 

\begin{lemma} \label{lem: cardinalitybound}
    Let $E \subseteq \R^d$ be an $m$ dimensional vector subspace for $1 \leq m \leq d-1$. For all $t > 0$ and $0<\theta < \frac{1}{2}$, we have
    \begin{equation*}
        \left|\left\{x \in \R^d : \frac{t}{2} \leq |x| \leq t, \; \frac{\theta}{2} \leq \Big|\operatorname{proj}_E \frac{x}{|x|}\Big| \leq \theta \right\}\right| \lesssim \theta^m t^d.
    \end{equation*}
\end{lemma}

Lemma~\ref{lem: cardinalitybound} easily follows by representing the set in polar coordinates. Using the hypothesis that the spheres in $\cC$ have $\delta$-separated centres, we are led to the following count. 

\begin{corollary}\label{cor: cardinality} For $3 \leq j \leq n$ and $C_1 \in \cC$, $C_2 \in \cC_{t_2}^{C_1}$ and $C_i \in \cC_{t_i, \theta_i}^{C_1, \hdots, C_{i-1}}$ for $3 \leq i \leq j - 1$, we have
\begin{equation*}
    \# \cC_{t_j, \theta_j}^{C_1, \hdots, C_{j-1}}  \lesssim \dfrac{\theta_j^{n-j+1}t_j^{n-1}}{\delta^{n-1}}.
\end{equation*}
\end{corollary}

\begin{proof} Recall that each $e(C_1, C_i)$ is a unit vector in $\R^{n-1}$. The hypotheses $C_2 \in \cC_{t_2}^{C_1}$ and $C_i \in \cC_{t_i, \theta_i}^{C_1, \hdots, C_{i-1}}$ for $3 \leq i \leq j-1$ ensure that $e(C_1, C_2), \dots, e(C_1, C_{j-1})$ are linearly independent. Thus, $E := \langle e(C_1, C_2), \hdots, e(C_1, C_{j-1}) \rangle^{\perp}$ is a vector subspace of $\R^{n-1}$ of dimension $(n-1) - (j-2) = n-j-1$. The result now follows by recalling the definition of $\cC_{t_j, \theta_j}^{C_1, \hdots, C_{j-1}}$ from \eqref{eq: pgnhole dist angle}, applying the volume bound from Lemma~\ref{lem: cardinalitybound} and using the hypothesis that the spheres in $\cC$ have $\delta$-separated centres.    
\end{proof}

%%%%%%%%%%%%%%%%%%%%%%%%%%%%%%%%%%%%%%%%%%%%%%%%%%%%%%%%%%%%%%%%%%%%%%%%%%%%%%%%%%%%%%%%%%%%%%%%

%                                          The main estimate

%%%%%%%%%%%%%%%%%%%%%%%%%%%%%%%%%%%%%%%%%%%%%%%%%%%%%%%%%%%%%%%%%%%%%%%%%%%%%%%%%%%%%%%%%%%%%%%%

\subsection{The main estimate} Combining the volume bound from Lemma~\ref{lem: generaltuplebound} with the cardinality bound from Corollary~\ref{cor: cardinality}, we arrive at the following estimate. 

\begin{proposition} \label{prop: angredhypothesis}
    Let $2 \leq m \leq n$ and $\delta \leq t_j \leq 1$ for $2 \leq j \leq m$ and $\frac{\delta}{t_j} \leq \theta_j \leq 1$ for $3 \leq j \leq m$. Then
\begin{equation}\label{eq: angredhypothesis}
 \sum_{C_1 \in \cC} \sum_{C_2 \in \cC^{C_1}_{t_2}} \sum_{C_3 \in \cC^{C_1, C_2}_{t_3, \theta_3}} \hdots \sum_{C_m \in \cC^{C_1, \hdots, C_{m-1}}_{t_m, \theta_m}} \big| \bigcap_{j=1}^{m} C_j^{\delta, \star} \big| \lesssim \delta^{a(m,n)} \prod_{j=3}^m \theta_j^{n-j} \prod_{j=2}^m t_j^{n-2} [\# \cC],   
\end{equation}
where $a(m, n) := m - (m-1)(n-1)$.
\end{proposition} 

Recall that our goal is to prove the multiplicity bound in Proposition~\ref{prop: multiplicity}, which can be restated as
\begin{equation}\label{eq: goal}
    \sum_{C_1, \dots, C_n \in \cC} \Big| \bigcap_{j=1}^{n} C_j^{\delta, \star} \Big| \lesssim(\log \delta^{-1}) \delta^{a(n,n)} [\# \cC],
\end{equation}
where the exponent $a(n,n) = n - (n-1)^2$ is as defined in Proposition~\ref{prop: angredhypothesis} above. If we fix $m = n$ in Proposition~\ref{prop: angredhypothesis} and sum the inequality \eqref{eq: angredhypothesis} over all dyadic choices of $\delta \leq t_j \leq 1$ for $2 \leq j \leq n$ and $\frac{\delta}{t_j} \leq \theta_j \leq 1$ for $3 \leq j \leq n$, then we come very close to completing the proof of Proposition~\ref{prop: multiplicity}.\footnote{Note, when we carry out this summation, we obtain a logarithmic loss either from summing in the $t_2$ parameter in the $n=2$ case or summing in the $\theta_n$ parameter in the $n \geq 3$ case.} Indeed, once we have done this, it only remains to account for terms on the left-hand side of \eqref{eq: goal} corresponding to degenerate configurations of spheres $C_1, \dots, C_n \in \cC$ where the centres are essentially affinely dependent (for instance, cases where $C_j \in \cC_{t_j, \leq \delta/t_j}^{C_1, \hdots, C_{j-1}}$ for some $3 \leq j \leq n$, where this set is as defined in \eqref{eq: angular degenerate}).\medskip

The remaining degenerate configurations are in fact easier to handle. For instance, if we consider the extreme case where all the centres essentially agree,
\begin{equation}\label{eq: base case}
     \sum_{\substack{C_1, \dots, C_n \in \cC \\ \dist(C_i, C_j) \lesssim \delta,\, 1 \leq i < j \leq n}} \Big| \bigcap_{j=1}^{n} C_j^{\delta, \star} \Big| \lesssim \delta [\# \cC],
\end{equation}
which is smaller than the right-hand side of \eqref{eq: goal}. Nevertheless, for completeness, in the next subsection we show how Proposition~\ref{prop: angredhypothesis} can be used to treat \textit{all} the terms on the left-hand side of \eqref{eq: goal}, including those corresponding to degenerate configurations, and thereby complete the proof of Proposition~\ref{prop: multiplicity}.

%%%%%%%%%%%%%%%%%%%%%%%%%%%%%%%%%%%%%%%%%%%%%%%%%%%%%%%%%%%%%%%%%%%%%%%%%%%%%%%%%%%%%%%%%%%%%%%%

%                                        Concluding the proof of Proposition~\ref{prop: multiplicity}

%%%%%%%%%%%%%%%%%%%%%%%%%%%%%%%%%%%%%%%%%%%%%%%%%%%%%%%%%%%%%%%%%%%%%%%%%%%%%%%%%%%%%%%%%%%%%%%%

\subsection{Concluding the proof of Proposition~\ref{prop: multiplicity}}

For the dedicated reader, here we conclude the proof of Proposition~\ref{prop: multiplicity} by repeated application of the bound from Proposition~\ref{prop: angredhypothesis}. This final part of the argument is a matter of technical bookkeeping: as mentioned, the heart of the proof is already contained in Proposition~\ref{prop: angredhypothesis}. 

\begin{proof}[Proof (of Proposition~\ref{prop: multiplicity})] Throughout what follows, we fix $0 < \delta < 1/2$. For $1 \leq m \leq n$ we shall prove by induction that
\begin{equation}\label{eq: induction hyp}
     \sum_{C_1, \dots, C_m \in \cC} \Big| \bigcap_{j=1}^m C_j^{\delta, \star} \Big| \lesssim (\log\delta^{-1})\delta^{n - (n-1)^2} [\# \cC],
\end{equation}
holds for all families $\cC$ of spheres with $\delta$-separated centres lying in $Q^{n-1} \times \{0\}$ and radii lying in $[1,2]$. Taking $m = n$ then gives Proposition~\ref{prop: multiplicity}.\medskip

For the base case $m = 1$, we obtain a stronger estimate by the same argument as used in \eqref{eq: base case}. Thus, we assume as an induction hypothesis that \eqref{eq: induction hyp} holds for some $1 \leq m \leq n-1$ and all families of spheres $\cC$ as above. Fixing $\cC$, our goal is now to prove\footnote{We shall obtain a different implicit constant at each stage of the induction, but since there are only $n$ steps the resulting constants are always admissible.}
\begin{equation}\label{eq: ind step}
     \sum_{C_1, \dots, C_{m+1} \in \cC} \Big| \bigcap_{j=1}^{m+1} C_j^{\delta, \star} \Big| \lesssim (\log\delta^{-1}) \delta^{n - (n-1)^2} [\# \cC].
\end{equation}

We may express the left-hand side of \eqref{eq: ind step} as
    \begin{equation}\label{eq: sep decomposition}
     \sum_{\substack{(C_1, \hdots, C_{m+1})\in\cC^{m+1} \\ \dist(C_{i_1}, C_{i_2}) < 2\delta \textrm{ for some }i_1 \neq i_2}} \big| \bigcap_{j=1}^{m+1} C_j^{\delta,\star} \big|+ \sum_{\substack{(C_1, \hdots, C_{m+1})\in \cC^{m+1} \\ \dist(C_{i_1}, C_{i_2}) \geq 2\delta \textrm{ for } i_1 \neq i_2}}\big| \bigcap_{j=1}^{m+1} C_j^{\delta,\star} \big|.
    \end{equation}
    The first term is easily treated using the $\delta$-separation of the centres and our induction hypothesis \eqref{eq: induction hyp}. It remains to bound the second term.\medskip

The above argument accounts for degenerate configurations with coincident centres. It remains to account for degenerate configurations with poor angular separation. Once we have done this, the remaining non-degenerate configurations can then be tackled by an appeal to Proposition~\ref{prop: angredhypothesis}.\medskip

Let $\bdT(\delta)$ denote the set of all tuples $\bdt = (t_j)_{j=2}^{m+1}$ where $t_j$ is a power of $2$ satisfying $\delta \leq t_j \leq 1$ for $2 \leq j \leq m+1$. Given $J \subseteq \{3, \hdots, m+1\}$ and $\bdt \in \bdT(\delta)$, let $\bdTheta_J(\delta; \bdt)$ denote the set of all tuples $\bdtheta = (\theta_j)_{j \in J}$ where $\theta_j$ is a power of $2$ satisfying $\frac{\delta}{t_j} \leq \theta_j \leq 1$ for $j \in J$.\medskip

Fix $J \subseteq \{3, \hdots, m+1\}$, $\bdt \in \bdT(\delta)$ and $\bdtheta \in \bdTheta_J(\delta; \bdt)$. Enumerate $\{1,2\} \cup J = \{\sigma(1), \hdots, \sigma(\ell)\}$ where $\ell = \#J + 2$ and $\sigma(i) < \sigma(i+1)$ for $1 \leq i \leq \ell-1$. Given $2 \leq j \leq m+1$, let $1 \leq i_j \leq \ell$ be the largest value satisfying $\sigma(i_j) \leq j - 1$. Let $\textbf{C}(\cC; \bdt, \bdtheta)$ denote the collection of all tuples $(C_1, \hdots, C_{m+1}) \in \cC^{m+1}$ which satisfy the following:
\begin{enumerate}[i)]
    \item $C_1 \in \cC$ and $C_2 \in \cC^{C_1}_{t_2}$;
    \item If $j \in J$, then $C_j \in \cC^{C_{\sigma(1)}, \hdots, C_{\sigma(i_j)}}_{t_j, \theta_j}$;
    \item If $j \in J^{\mathrm{c}} := \{3, \hdots, m+1\} \setminus J$, then $C_j \in \cC^{C_{\sigma(1)}, \hdots, C_{\sigma(i_j)}}_{t_j, \leq \delta/t_j}$.
\end{enumerate}
Here the set of spheres appearing in iii) is as defined in \eqref{eq: angular degenerate}. Letting $\Sigma_{\mathrm{sep}}$ denote the second term in \eqref{eq: sep decomposition}, it follows that
\begin{equation} \label{angred1}
    \Sigma_{\mathrm{sep}} = \sum_{J \subseteq \{3, \hdots, m+1\}} \sum_{\bdt \in \bdT(\delta)} \sum_{\bdtheta \in \bdTheta_J(\delta; \bdt)} \sum_{(C_1, \hdots, C_{m+1}) \in \textbf{C}(\cC; \bdt, \bdtheta)} \big| \bigcap_{j=1}^{m+1} C_j^{\delta,\star} \big|.
\end{equation}

We now define the reduced family $\textbf{C}_{\text{red}}(\cC; \bdt, \bdtheta)$ to be the collection of all $\ell$-tuples $(C_{\sigma(1)}, \hdots, C_{\sigma(\ell)}) \in \cC^{\ell}$ satisfying
\begin{enumerate}[i)]
    \item $C_1 \in \cC$ and $C_2 \in \cC^{C_1}_{t_2}$;
    \item If $3 \leq i \leq \ell$, then we have $C_{\sigma(i)} \in \cC^{C_{\sigma(1)}, \hdots, C_{\sigma(i-1)}}_{t_{\sigma(i)}, \theta_{\sigma(i)}}$.
\end{enumerate}
If $(C_1, \hdots, C_{m+1})\in \textbf{C}(\cC; \bdt, \bdtheta)$, then $(C_{\sigma(1)}, \hdots, C_{\sigma(\ell)})\in\textbf{C}_{\text{red}}(\cC; \bdt, \bdtheta)$ and
\begin{equation*}
\big| \bigcap_{j=1}^{m+1} C_j^{\delta, \star} \big| \leq \big| \bigcap_{i=1}^{\ell} C_{\sigma(i)}^{\delta, \star} \big|,
\end{equation*}
so that 
\begin{equation} \label{angred2}
\sum_{(C_1, \hdots, C_{m+1}) \in \textbf{C}(\cC; \bdt, \bdtheta)} \big| \bigcap_{j=1}^{m+1} C_j^{\delta, \star} \big| \leq M_J(\delta;\mathbf{t}) \sum_{(C_{\sigma(1)}, \hdots, C_{\sigma(\ell)}) \in \textbf{C}_{\text{red}}(\cC; \bdt, \bdtheta)} \big| \bigcap_{i=1}^{\ell} C_{\sigma(i)}^{\delta, \star} \big|
\end{equation}
where
\begin{equation*}
    M_J(\delta;\mathbf{t}) := \prod_{j \in J^{\mathrm{c}}} \max_{C_{\sigma(1)}, \hdots, C_{\sigma(i_j)} \in \cC} \, \big[\#\cC^{C_{\sigma(1)}, \hdots, C_{\sigma(i_j)}}_{t_j, \leq \delta/t_j}\big].
\end{equation*}
Using Proposition~\ref{prop: angredhypothesis}, we have 
\begin{equation} \label{angred3}
    \sum_{(C_{\sigma(1)}, \hdots, C_{\sigma(\ell)}) \in \textbf{C}_{\text{red}}(\cC; \bdt, \bdtheta)} \big| \bigcap_{i=1}^{\ell} C_{\sigma(i)}^{\delta, \star} \big| \lesssim \delta^{a(\ell,n)} \prod_{j \in J} \theta_j^{n - j} \prod_{j \in \{2\} \cup J} t_j^{n-2} [\#\cC].
\end{equation}
where, recall, $a(\ell, n) := \ell - (\ell-1)(n-1)$. Here the exponent $n - j$ of $\theta_j$ perhaps requires some explanation. Direct application of Proposition~\ref{prop: angredhypothesis} results in the exponent $n - \sigma^{-1}(j)$ for $j \in J$, where here we consider $\sigma$ a bijection from $\{1, \dots, \ell\}$ to $\{1,2\} \cup J$. However, it is clear from the definitions that $\sigma^{-1}(j) \leq j$ and, since $\theta_j \lesssim 1$, we can pass to $n - j$ as above.\medskip

If $C_j \in \cC^{C_{\sigma(1)}, \hdots, C_{\sigma(i_j)}}_{t_j, \leq \delta/t_j}$, then the centre $x(C_j)$ of $C_j$ must lie in the ball $B(x(C_1), t_j)$ in an $O(\delta)$ neighbourhood of an $i_j-1$ dimensional plane spanned by vectors $\bde(C_{\sigma(1)}, C_{\sigma(i)})$ for $2 \leq i \leq i_j$. Therefore,
\begin{equation*}
\# \cC^{C_{\sigma(1)}, \hdots, C_{\sigma(i_j)}}_{t_j, \leq \delta/t_j} \lesssim \Big(\frac{t_j}{\delta}\Big)^{i_j-1}, 
\end{equation*}
so that
\begin{equation*}
    M_J(\delta;\mathbf{t}) \lesssim \Big(\prod_{j \in J^{\mathrm{c}} } t_j^{i_j-1}\Big) \delta^{-\sum_{j\in J^{\mathrm{c}}}(i_j-1)}.
\end{equation*}
Now since $i_j \geq 2$ for all $j\in J^{\mathrm{c}}$ and $t_j \lesssim 1$, we can bound $t_j^{i_j-1} \lesssim t_j$. Since we also have $i_j \leq m$ for all $j\in J^{\mathrm{c}}$ and $\#J^{\mathrm{c}} = (m+1-2)-(\ell-2)= m+1-\ell$, we have
\begin{equation*}
    \sum_{j\in J^{\mathrm{c}}}(i_j-1) \leq (m-1)(m+1 - \ell) \leq (n-2)(n - \ell).
\end{equation*}
Therefore,
\begin{equation} \label{angred4}
    M_J(\delta;\mathbf{t}) \lesssim \delta^{-(n-2)(n - \ell)} \prod_{j\in J^{\mathrm{c}} } t_j. 
\end{equation}
Substituting \eqref{angred3} and \eqref{angred4} into \eqref{angred2} and simplifying the exponents, we obtain
\begin{equation*}
    \sum_{(C_1, \hdots, C_{m+1}) \in \textbf{C}(\cC; \bdt, \bdtheta)} \big| \bigcap_{j=1}^{m+1} C_j^{\delta, \star} \big| \lesssim \delta^{n-(n-1)^2} \Big(\prod_{j \in J} \theta_j^{n - j}\Big) \Big( \prod_{j=2}^{m+1} t_j^{\min\{n-2, 1\}}  \Big) [\#\cC].
\end{equation*}
Applying this termwise to \eqref{angred1} and summing over the dyadic parameters $\bdt \in \bdT(\delta)$ and $\bdtheta \in \bdTheta_J(\delta; \bdt)$ and $J \subseteq \{3, \hdots, n\}$, we conclude that 
\begin{equation*}
    \Sigma_{\mathrm{sep}} \lesssim (\log\delta^{-1})\delta^{n-(n-1)^2}\#\cC.
\end{equation*}
Here the logarithm arises in one of two ways:\smallskip

 \noindent If $n = 2$ and $m = 1$, then there is only a single distance parameter $t_2$ and no angular parameters. In this case, $\min\{n-2, 1\} = 0$ and summing in $t_2$ incurs a logarithm. \smallskip
 
\noindent If $n \geq 3$, then $\min\{n-2, 1\} = 1$ and so we can sum in all the $t_j$ parameters without incurring any loss. Similarly, we can sum in $\theta_j$ for $3 \leq j \leq n-1$ without incurring any loss. Thus, we only lose a logarithm when $n \in J$, from the summation in $\theta_n$. \smallskip

We have obtained favourable bounds for both terms \eqref{eq: sep decomposition} and, consequently, \eqref{eq: ind step} holds. This closes the induction and concludes the proof.   
\end{proof}

%%%%%%%%%%%%%%%%%%%%%%%%%%%%%%%%%%%%%%%%%%%%%%%%%%%%%%%%%%%%%%%%%%%%%%%%%%%%%%%%%%%%%%%%%%%%%%%%

%                                          REFERENCES

%%%%%%%%%%%%%%%%%%%%%%%%%%%%%%%%%%%%%%%%%%%%%%%%%%%%%%%%%%%%%%%%%%%%%%%%%%%%%%%%%%%%%%%%%%%%%%%%

\bibliography{Reference}

\providecommand{\bysame}{\leavevmode\hbox to3em{\hrulefill}\thinspace}
\providecommand{\MR}{\relax\ifhmode\unskip\space\fi MR }
\providecommand{\MRhref}[2]{%
  \href{http://www.ams.org/mathscinet-getitem?mr=#1}{#2}
}
\providecommand{\href}[2]{#2}
\begin{thebibliography}{10}

\bibitem{BGHS}
David Beltran, Shaoming Guo, Jonathan Hickman, and Andreas Seeger, \emph{{S}harp {$L^p$} bounds for the helical maximal function}, to appear \textit{Amer. J. Math.}

\bibitem{Bourgain1986}
J.~Bourgain, \emph{Averages in the plane over convex curves and maximal operators}, J. Analyse Math. \textbf{47} (1986), 69--85. \MR{874045}

\bibitem{Carbery1988}
Anthony Carbery, \emph{Covering lemmas revisited}, Proc. Edinburgh Math. Soc. (2) \textbf{31} (1988), no.~1, 145--150. \MR{930022}

\bibitem{CDK}
Alan Chang, Georgios Dosidis, and Jongchon Kim, \emph{{N}ikodym sets and maximal functions associated with spheres}, to appear \textit{Rev. Mat. Iberoam.}

\bibitem{CGY}
Mingfeng Chen, Shaoming Guo, and Tongou Yang, \emph{{A} multi{-}parameter cinematic curvature}, preprint: \verb+arXiv:2306.01606+.

\bibitem{Cordoba1977}
Antonio Cordoba, \emph{The {K}akeya maximal function and the spherical summation multipliers}, Amer. J. Math. \textbf{99} (1977), no.~1, 1--22. \MR{447949}

\bibitem{KLO2022}
Hyerim Ko, Sanghyuk Lee, and Sewook Oh, \emph{Maximal estimates for averages over space curves}, Invent. Math. \textbf{228} (2022), no.~2, 991--1035. \MR{4411734}

\bibitem{KW1999}
Lawrence Kolasa and Thomas Wolff, \emph{On some variants of the {K}akeya problem}, Pacific J. Math. \textbf{190} (1999), no.~1, 111--154. \MR{1722768}

\bibitem{KST1954}
T.~K\"ovari, V.~T. S\'os, and P.~Tur\'an, \emph{On a problem of {K}. {Z}arankiewicz}, Colloq. Math. \textbf{3} (1954), 50--57. \MR{65617}

\bibitem{LLO2025}
Juyoung Lee, Sanghyuk Lee, and Sewook Oh, \emph{The elliptic maximal function}, J. Funct. Anal. \textbf{288} (2025), no.~1, Paper No. 110693, 31. \MR{4800578}

\bibitem{Marstrand1987}
J.~M. Marstrand, \emph{Packing circles in the plane}, Proc. London Math. Soc. (3) \textbf{55} (1987), no.~1, 37--58. \MR{887283}

\bibitem{Mattila_book}
Pertti Mattila, \emph{Fourier analysis and {H}ausdorff dimension}, Cambridge Studies in Advanced Mathematics, vol. 150, Cambridge University Press, Cambridge, 2015. \MR{3617376}

\bibitem{MSS1992}
Gerd Mockenhaupt, Andreas Seeger, and Christopher~D. Sogge, \emph{Wave front sets, local smoothing and {B}ourgain's circular maximal theorem}, Ann. of Math. (2) \textbf{136} (1992), no.~1, 207--218. \MR{1173929}

\bibitem{R1986}
Jos\'e{}~L. Rubio~de Francia, \emph{Maximal functions and {F}ourier transforms}, Duke Math. J. \textbf{53} (1986), no.~2, 395--404. \MR{850542}

\bibitem{Schlag1997}
W.~Schlag, \emph{A generalization of {B}ourgain's circular maximal theorem}, J. Amer. Math. Soc. \textbf{10} (1997), no.~1, 103--122. \MR{1388870}

\bibitem{Schlag1998}
\bysame, \emph{A geometric proof of the circular maximal theorem}, Duke Math. J. \textbf{93} (1998), no.~3, 505--533. \MR{1626711}

\bibitem{Stein1976}
Elias~M. Stein, \emph{Maximal functions. {I}. {S}pherical means}, Proc. Nat. Acad. Sci. U.S.A. \textbf{73} (1976), no.~7, 2174--2175. \MR{420116}

\bibitem{Wolff2000}
T.~Wolff, \emph{Local smoothing type estimates on {$L^p$} for large {$p$}}, Geom. Funct. Anal. \textbf{10} (2000), no.~5, 1237--1288. \MR{1800068}

\bibitem{Wolff1999}
Thomas Wolff, \emph{A {K}akeya-type problem for circles}, Amer. J. Math. \textbf{119} (1997), no.~5, 985--1026. \MR{1473067}

\bibitem{Zahl}
Joshua Zahl, \emph{{O}n maximal functions associated to families of curves in the plane}, preprint: \verb+arXiv:2307.05894+.

\end{thebibliography}
\bibliographystyle{amsplain}

\end{document}